\documentclass[11pt]{article}    

\usepackage[margin=0.75in]{geometry} 
\usepackage{amsmath,amssymb,amsthm}

\newtheorem{theorem}{Theorem}[section]

\newtheorem{lemma}[theorem]{Lemma}

\newtheorem{remark}{Remark}

\newcommand{\al}{\alpha}
\newcommand{\bt}{\beta}

\newcommand{\s}{\sigma}

\newcommand{\be}{\begin{equation}}
\newcommand{\ee}{\end{equation}}
\newcommand{\bea}{\begin{eqnarray}}
\newcommand{\eea}{\end{eqnarray}}

\numberwithin{equation}{section}
\begin{document}
\title{Painlev\'{e} IV, Chazy II, and Asymptotics for Recurrence Coefficients of Semi-classical Laguerre Polynomials and Their Hankel Determinants}
\author{Chao Min\thanks{School of Mathematical Sciences, Huaqiao University, Quanzhou 362021, China; e-mail: chaomin@hqu.edu.cn}\: and Yang Chen\thanks{Department of Mathematics, Faculty of Science and Technology, University of Macau, Macau, China; e-mail:\qquad\qquad\qquad\qquad yangbrookchen@yahoo.co.uk}}


\date{\today}
\maketitle
\begin{abstract}
This paper studies the monic semi-classical Laguerre polynomials based on previous work by Boelen and Van Assche \cite{Boelen}, Filipuk et al. \cite{Filipuk} and Clarkson and Jordaan \cite{Clarkson}. Filipuk, Van Assche and Zhang proved that the diagonal recurrence coefficient $\alpha_n(t)$ satisfies the fourth Painlev\'{e} equation. In this paper we show that the off-diagonal recurrence coefficient $\beta_n(t)$ fulfills the first member of Chazy II system. We also prove that the sub-leading coefficient of the monic semi-classical Laguerre polynomials satisfies both the continuous and discrete Jimbo-Miwa-Okamoto $\sigma$-form of Painlev\'{e} IV.
By using Dyson's Coulomb fluid approach together with the discrete system for $\alpha_n(t)$ and $\beta_n(t)$, we obtain the large $n$ asymptotic expansions of the recurrence coefficients and the sub-leading coefficient. The large $n$ asymptotics of the associate Hankel determinant (including the constant term) is derived from its integral representation in terms of the sub-leading coefficient.
\end{abstract}

$\mathbf{Keywords}$: Semi-classical Laguerre polynomials; Recurrence coefficients; Painlev\'{e} IV;

Chazy II system; Hankel determinants; Asymptotic expansions.

$\mathbf{Mathematics\:\: Subject\:\: Classification\:\: 2020}$: 42C05, 33E17, 41A60.

\section{Introduction}
Orthogonal polynomials play an important role in mathematical physics (e.g., random matrix theory, integrable systems), approximation theory, mechanical quadrature, etc. The relationship between recurrence coefficients of semi-classical orthogonal polynomials and Painlev\'{e} equations has been studied extensively over the past decade; see \cite{Boelen,ChenIts,Clarkson,Clarkson3,Dai,Filipuk,Min2022,VanAssche} for reference. Semi-classical orthogonal polynomials are defined as orthogonal polynomials whose weight functions $w(x)$ satisfy the Pearson equation
\be\label{pe}
\frac{d}{dx}(\varrho(x)w(x))=\tau(x)w(x),
\ee
where $\varrho(x)$ and $\tau(x)$ are polynomials with deg $\varrho>2$ or deg $\tau\neq 1$ \cite[Section 1.1.1]{VanAssche}.

Boelen and Van Assche \cite{Boelen} considered the \textit{orthonormal} polynomials with respect to the so-called semi-classical Laguerre weight
\be\label{weight}
w(x)=w(x;t):=x^\lambda\mathrm{e}^{-x^2+tx},\qquad x\in \mathbb{R}^+
\ee
with $\lambda> -1,\;t\in \mathbb{R}$. It is easy to check that (\ref{weight}) is indeed a semi-classical weight since it satisfies the Pearson equation (\ref{pe}) with
$$
\varrho(x)=x,\qquad \tau(x)=-2x^2+tx+1+\lambda.
$$

It was shown in \cite[Theorem 1.1]{Boelen} that the recurrence coefficients of the semi-classical Laguerre polynomials satisfy a discrete system, which is related to an asymmetric discrete Painlev\'{e} IV equation. Later, Filipuk, Van Assche and Zhang \cite[Theorem 1.1]{Filipuk} proved that the recurrence coefficient $b_n(t)$, which is equal to $\al_n(t)$ below, satisfies the (continuous) Painlev\'{e} IV equation.

More recently, Clarkson and Jordaan \cite{Clarkson} studied the \textit{monic} orthogonal polynomials with respect to the weight (\ref{weight}), i.e.,
\be\label{or}
\int_{0}^{\infty}P_{m}(x;t)P_{n}(x;t)w(x;t)dx=h_{n}(t)\delta_{mn},\qquad m, n=0,1,2,\ldots.
\ee
Here $P_{n}(x;t)$ has the monomial expansion
\be\label{expan}
P_{n}(x;t)=x^{n}+\mathrm{p}(n,t)x^{n-1}+\cdots+P_{n}(0;t),\qquad n=0,1,2,\ldots,
\ee
and $\mathrm{p}(n,t)$ denotes the coefficient of $x^{n-1}$ and we set $\mathrm{p}(0,t)=0$.

It is well known that the orthogonal polynomials $P_{n}(x;t)$ obey the three-term recurrence relation of the form \cite{Szego}
\be\label{rr}
xP_{n}(x;t)=P_{n+1}(x;t)+\al_n(t)P_{n}(x;t)+\beta_{n}(t)P_{n-1}(x;t),
\ee
with the initial conditions
$$
P_{0}(x;t)=1,\qquad \beta_{0}(t)P_{-1}(x;t)=0.
$$
The combination of (\ref{or}), (\ref{expan}) and (\ref{rr}) gives
\be\label{al}
\al_{n}(t)=\mathrm{p}(n,t)-\mathrm{p}(n+1,t),
\ee
\be\label{be}
\beta_{n}(t)=\frac{h_{n}(t)}{h_{n-1}(t)}.
\ee
Taking a telescopic sum of (\ref{al}), we have
\be\label{sum}
\sum_{j=0}^{n-1}\al_{j}(t)=-\mathrm{p}(n,t).
\ee

Based on the results in \cite{Boelen} and \cite{Filipuk}, Clarkson and Jordaan \cite{Clarkson} showed the following two lemmas.
\begin{lemma}\label{lemma}
The recurrence coefficients $\al_n(t)$ and $\bt_n(t)$ satisfy the discrete system:
\begin{subequations}\label{ab3}
\be
\al_n(2\al_n-t)+2\bt_n+2\bt_{n+1}=2n+1+\lambda,
\ee
\be
(2\al_n-t)(2\al_{n-1}-t)\bt_n=(2\bt_n-n)(2\bt_n-n-\lambda).
\ee
\end{subequations}
\end{lemma}

\begin{lemma}\label{le}
The recurrence coefficients $\al_n(t)$ and $\bt_n(t)$ are given by
$$
\al_n(t)=\frac{1}{2}q_n(s)+\frac{1}{2}t,
$$
$$
\bt_n(t)=-\frac{1}{8}q_n'(s)-\frac{1}{8}q_n^2(s)-\frac{1}{4}s\:q_n(s)+\frac{1}{2}n+\frac{1}{4}\lambda,
$$
with $s=\frac{1}{2}t$, where $q_n(s)$ satisfies the fourth Painlev\'{e} equation \cite{Gromak}
\be\label{p4}
q_n''(s)=\frac{(q_n'(s))^2}{2q_n(s)}+\frac{3}{2}q_n^3(s)+4s\:q_n^2(s)+2(s^2-2n-1-\lambda)q_n(s)-\frac{2\lambda^2}{q_n(s)}.
\ee
\end{lemma}

The Hankel determinant generated by the weight (\ref{weight}) is defined by
$$
D_{n}(t):=\det(\mu_{i+j}(t))_{i,j=0}^{n-1}=\begin{vmatrix}
\mu_{0}(t)&\mu_{1}(t)&\cdots&\mu_{n-1}(t)\\
\mu_{1}(t)&\mu_{2}(t)&\cdots&\mu_{n}(t)\\
\vdots&\vdots&&\vdots\\
\mu_{n-1}(t)&\mu_{n}(t)&\cdots&\mu_{2n-2}(t)
\end{vmatrix}
,
$$
where $\mu_{j}(t)$ is the $j$th moment given by
\bea
\mu_{j}(t):&=&\int_{0}^{\infty}x^{j}w(x;t)dx\nonumber\\
&=&\frac{1}{2}\left[\Gamma\left(\frac{j+1+\lambda}{2}\right){}_{1}F_{1}\left(\frac{j+1+\lambda}{2};\frac{1}{2};\frac{t^2}{4}\right)
+t\:\Gamma\left(\frac{j+2+\lambda}{2}\right){}_{1}F_{1}\left(\frac{j+2+\lambda}{2};\frac{3}{2};\frac{t^2}{4}\right)\right].\nonumber
\eea
The moments can also be expressed in terms of the parabolic cylinder functions \cite{Clarkson}.

It is a known fact that the Hankel determinant can be expressed as a product of the square of $L^2$ norms for the monic orthogonal polynomials \cite[(2.1.6)]{Ismail},
\be\label{hankel}
D_{n}(t)=\prod_{j=0}^{n-1}h_{j}(t).
\ee
In addition, it was shown in \cite{Clarkson} that the Hankel determinant $D_n(t)$, generated by the semi-classical Laguerre weight, satisfies the Toda molecule equation \cite{Sogo}
$$
\frac{d^2}{dt^2}\ln D_n(t)=\frac{D_{n+1}(t)D_{n-1}(t)}{D_n^2(t)}.
$$

The rest of the paper is organized as follows. In the next subsection we give a summary of the main results obtained in this paper. In Section 2, we recall some important results of the paper by Filipuk, Van Assche and Zhang \cite{Filipuk}, from which we derive the second-order differential equation for the semi-classical Laguerre polynomials. In Section 3, we show that the auxiliary quantities $R_n(t)$ and $r_n(t)$ satisfy the coupled Riccati equations. This enables us to obtain the second-order differential equation for $\bt_n(t)$, which is equivalent to a Chazy type equation under the suitable transformation. Furthermore, we find that $\mathrm{p}(n,t)$, the sub-leading coefficient of the monic semi-classical Laguerre polynomials, satisfies both the continuous and discrete Jimbo-Miwa-Okamoto $\s$-form of Painlev\'{e} IV. In Section 4, we derive the large $n$ asymptotics of the recurrence coefficients, the sub-leading coefficient and the Hankel determinant by using the Coulomb fluid approach.
\subsection{Statement of Main Results}
In this subsection, we present the main results obtained in this paper, which are not considered in previous work \cite{Boelen,Clarkson,Filipuk}. For convenience, we will take $\lambda$ in the weight (\ref{weight}) to be strictly positive in the following discussions. This is due to two reasons. Firstly, it makes the weight vanish at the endpoints of the orthogonality interval and then the ladder operator approach can be applied. Secondly, in this case the potential for the weight is \textit{convex} such that the equilibrium density discussed in Section 4 is supported in a single interval (the so-called one-cut case).

For brevity, we will not show the $t$-dependence of all the quantities, such as the recurrence coefficients $\al_n$ and $\bt_n$, considered in this paper from now on. By applying the ladder operators to the monic semi-classical Laguerre polynomials, we have the following theorem.
\begin{theorem}\label{thm1}
The semi-classical Laguerre polynomials $P_n(x)$ satisfy the linear second-order ordinary differential equation
\be\label{odep}
P_n''(x)+\Phi(x) P_n'(x)+\Psi(x) P_n(x)=0,
\ee
where $\Phi(x)$ and $\Psi(x)$ are expressed in terms of $\al_n$ and $\bt_n$ as follows:
\be\label{phi}
\Phi(x)=t-2x+\frac{\lambda}{x}-\frac{t-2 \alpha_n}{x (2x-t+2 \alpha_n)},
\ee
\be\label{psi}
\Psi(x)=2n-\frac{nt-4\al_n\bt_n}{x}-\frac{2(n-2\bt_n)(n+\lambda-2\bt_n)}{x(t-2\al_n)}+\frac{n-2 \beta_n}{x^2}+\frac{(t-2 \alpha_n)(n-2 \beta_n) }{x^2 (2x-t+2 \alpha_n)}.
\ee
\end{theorem}
\begin{theorem}\label{thm2}
The recurrence coefficient $\bt_n(t)$ satisfies the second-order differential equation
\be\label{bde}
\left[ 2\beta_n ''+12 \beta_n^2-4 (2 n+\lambda) \beta_n+n (n+\lambda)\right]^2=t^2 \left[(\beta_n ')^2+4 \beta_n^3-2 (2 n+\lambda) \beta_n^2+n (n+\lambda) \beta_n\right].
\ee
Let $t=\sqrt{2}\:z$ and
$$
\bt_n(t)=\frac{2n+\lambda}{6}-\frac{v(z)}{2}.
$$
Then $v(z)$ admits the first member of Chazy II system \cite[(1.17)]{Cosgrove}
\be\label{ce}
\big(v''-6v^2-\tilde{\al}_1\big)^2=z^2\big((v')^2-4v^3-2\tilde{\al}_1 v-\tilde{\bt}_1\big),
\ee
with the parameters
\be\label{val}
\tilde{\al}_1=-\frac{2}{3}(n^2+n\lambda+\lambda^2),\qquad \tilde{\bt}_1=-\frac{4}{27}(2n^3+3n^2\lambda-3n\lambda^2-2\lambda^3).
\ee
\end{theorem}
The Chazy equations were first found by Chazy \cite{Chazy1909,Chazy1911} and subsequently derived by a number of authors. These equations also arose in the problems on the Gaussian, Laguerre and Jacobi weights with jump discontinuities \cite{Lyu2017,Min2019,Witte}.

The following theorem reveals the relation between the sub-leading coefficient $\mathrm{p}(n,t)$ and the Painlev\'{e} IV equation.
\begin{theorem}\label{thm3}
Let
\be\label{sn}
\s_n(s):=-2\mathrm{p}(n,t)-(n+\lambda)t,\qquad\qquad s=\frac{1}{2}t.
\ee
Then $\s_n(s)$ satisfies the Jimbo-Miwa-Okamoto $\s$-form of Painlev\'{e} IV \cite[(C.37)]{Jimbo1981}:
\be\label{jmo}
(\s_n''(s))^2=4(s\s_n'(s)-\s_n(s))^2-4(\s_n'(s)+\nu_0)(\s_n'(s)+\nu_1)(\s_n'(s)+\nu_2),
\ee
with the parameters
$$
\nu_0=0,\qquad \nu_1=2\lambda,\qquad \nu_2=2n+2\lambda.
$$
Moreover, $\s_n(s)$ also admits the discrete $\s$-form of Painlev\'{e} IV:
\bea\label{dis}
&&2\left[\s_n+n(\s_{n-1}-\s_{n+1})+2\lambda s\right]\left[\s_n+(n+\lambda)(\s_{n-1}-\s_{n+1})\right]\nonumber\\
&=&\left[\s_n+2(n+\lambda)s\right](\s_{n+1}-\s_{n-1}+2s)(\s_{n-1}-\s_{n})(\s_n-\s_{n+1}).
\eea
\end{theorem}

The following results are concerned with the large $n$ asymptotics of the recurrence coefficients $\al_n$ and $\bt_n$, the sub-leading coefficient $\mathrm{p}(n,t)$ and the Hankel determinant $D_n(t)$. The derivation is based on the Coulomb fluid approach together with the discrete system satisfied by the recurrence coefficients.
\begin{theorem}\label{abt}
The recurrence coefficients $\al_n$ and $\bt_n$ have the asymptotic expansions as $n\rightarrow\infty$:
\bea\label{aln}
\al_n&=&\sqrt{\frac{2n}{3}}+\frac{t}{6}+\frac{t^2+12 (1+\lambda)}{24 \sqrt{6n} }-\frac{t^4+24 t^2(1+\lambda) -48 (6 \lambda ^2-6 \lambda -5)}{2304 \sqrt{6}\: n^{3/2}}\nonumber\\[10pt]
&+&\frac{t(9 \lambda ^2-2)}{144 n^2}+\frac{t^6+36 t^4(1+\lambda) +144 t^2(66 \lambda ^2+6 \lambda -13) -1728 (8 \lambda ^3+6 \lambda ^2-5 \lambda -3)}{110592 \sqrt{6}\: n^{5/2}}\nonumber\\[10pt]
&+&\frac{t \left[t^2(27 \lambda ^2-7) -12 (9 \lambda ^3+9 \lambda ^2-2 \lambda -2)\right]}{1728 n^3}+O(n^{-7/2}),
\eea
\bea\label{ben}
\bt_n&=&\frac{n}{6}+\frac{t\sqrt{n}}{6 \sqrt{6}}+\frac{t^2+6 \lambda}{72}+\frac{t (t^2+12 \lambda)}{288 \sqrt{6n} }+\frac{2-9 \lambda ^2}{144 n}-\frac{t (t^4+24 \lambda  t^2+3168 \lambda ^2-816)}{27648 \sqrt{6}\: n^{3/2}}\nonumber\\[10pt]
&+&\frac{t^2(7-27 \lambda ^2) +4 \lambda (9 \lambda ^2-2)}{1152 n^2}+\frac{t \left[t^6+36 \lambda  t^4-144 t^2(246 \lambda ^2-61)+1728 \lambda  (64 \lambda ^2-17)\right]}{1327104 \sqrt{6}\: n^{5/2}}\nonumber\\[10pt]
&+&O(n^{-3}).
\eea
\end{theorem}
\begin{theorem}\label{hd}
The sub-leading coefficient $\mathrm{p}(n,t)$ has the large $n$ asymptotic expansion
\bea\label{hne}
\mathrm{p}(n,t)&=&-\frac{2}{3}\sqrt{\frac{2}{3}}\:n^{3/2}-\frac{nt}{6}-\frac{(t^2+12\lambda)\sqrt{n}}{12\sqrt{6}}-\frac{t(t^2+18\lambda)}{216}-\frac{t^4+24\lambda t^2-288\lambda^2+48}{1152\sqrt{6n}}\nonumber\\[10pt]
&+&\frac{t(9\lambda^2-2)}{144n}+\frac{t^6+36\lambda t^4+144t^2(66\lambda^2-17)-1728\lambda(8\lambda^2-1)}{165888\sqrt{6}\:n^{3/2}}\nonumber\\[10pt]
&+&\frac{t\left[t^2(27\lambda^2-7)+12\lambda(2-9\lambda^2)\right]}{3456n^2}+O(n^{-5/2}).
\eea
\end{theorem}

\begin{theorem}\label{hkd}
The Hankel determinant $D_n(t)$ has the large $n$ expansion
\bea
\ln D_n(t)&=&\frac{1}{2} n^2 \ln n-\frac{3+2\ln 6}{4}  n^2+\frac{2}{3} \sqrt{\frac{2}{3}}\: n^{3/2}t+\frac{\lambda}{2}     n \ln n+C_1 n\nonumber\\[10pt]
&+&\frac{t (t^2+36 \lambda)\sqrt{n}}{36 \sqrt{6}}+\frac{3 \lambda ^2-1}{6} \ln n +C_2+\frac{t \left[t^4+40 \lambda  t^2+240 (1-6 \lambda ^2)\right]}{5760 \sqrt{6n}}\nonumber\\[10pt]
&-&\frac{(9 \lambda ^2-2) t^2-12 \lambda(5\lambda^2-2)}{288 n}+O(n^{-3/2}),\nonumber
\eea
where
$$
C_1=\frac{t^2}{12}+\ln (2 \pi )-\frac{\lambda  (1+\ln 6)}{2},
$$
$$
C_2=\frac{t^2(t^2+36 \lambda)  }{864}+\frac{1}{24} \left[48 \zeta'(-1)-24 \ln G(\lambda +1)-12 \lambda ^2 \ln \frac{3}{2}+12 \lambda  \ln (2 \pi )-4 \ln 2+3\ln 3\right]
$$
and $\zeta(\cdot)$ is the Riemann zeta function and $G(\cdot)$ is the Barnes $G$-function which satisfies the relation \cite{Barnes,Voros}
$$
G(z+1)=\Gamma(z)G(z),\qquad\qquad G(1):=1.
$$
\end{theorem}
Finally, we make a remark about the differential equation (\ref{odep}) as $n\rightarrow\infty$. Substituting the large $n$ expansion of $\al_n$ and $\bt_n$ in Theorem \ref{abt} into (\ref{phi}) and (\ref{psi}), we obtain
$$
\Phi(x)=t-2x+\frac{1+\lambda}{x}+O(n^{-1/2}),
$$
$$
\Psi(x)=\frac{4\sqrt{6}\:n^{3/2}}{9x}+O(n).
$$
Considering the equation
$$
\tilde{P}_n''(x)+\left(t-2x+\frac{1+\lambda}{x}\right)\tilde{P}_n'(x)+\frac{4\sqrt{6}\:n^{3/2}}{9x}\tilde{P}_n(x)=0,
$$
one would find that this is the biconfluent Heun equation (BHE) \cite[p. 194 (1.2.5)]{Ronveaux}. The relations between orthogonal polynomials and Heun's differential equations have been discussed in recent years; see \cite{Alhaidari,CFZ,MNR} for reference.
\section{Ladder Operators and Second-order Differential Equation}
The ladder operator approach has been applied to solve problems on orthogonal polynomials for many years. This approach is especially useful to establish the relations between Painlev\'{e} equations and recurrence coefficients of semi-classical orthogonal polynomials. See \cite{ChenIts,Dai,Filipuk,Min2022,VanAssche} for reference.

Following the general set-up (see, e.g., \cite{ChenIts}) and noting that $w(0)=w(\infty)=0$ since we require $\lambda>0$ in (\ref{weight}), Filipuk, Van Assche and Zhang \cite{Filipuk} showed that the monic semi-classical Laguerre polynomials $P_n(x)$ satisfy the following ladder operator equations:
\be\label{ra}
\left(\frac{d}{dx}+B_{n}(x)\right)P_{n}(x)=\beta_{n}A_{n}(x)P_{n-1}(x),
\ee
\be\label{lo}
\left(\frac{d}{dx}-B_{n}(x)-\mathrm{v}'(x)\right)P_{n-1}(x)=-A_{n-1}(x)P_{n}(x),
\ee
where $\mathrm{v}(x)$ is the potential
\be\label{po}
\mathrm{v}(x)=-\ln w(x)=x^2-tx-\lambda \ln x
\ee
and
\be\label{an}
A_{n}(x)=2+\frac{R_n(t)}{x},
\ee
\be\label{bn}
B_{n}(x)=\frac{r_n(t)}{x},
\ee
with
$$
R_{n}(t):=\frac{\lambda}{h_{n}}\int_{0}^{\infty}P_{n}^2(y)y^{\lambda-1}\mathrm{e}^{-y^2+ty}dy,
$$
$$
r_{n}(t):=\frac{\lambda}{h_{n-1}}\int_{0}^{\infty}P_{n}(y)P_{n-1}(y)y^{\lambda-1}\mathrm{e}^{-y^2+ty}dy.
$$

Substituting (\ref{an}) and (\ref{bn}) into the compatibility conditions for the ladder operators, Filipuk et al. \cite{Filipuk} obtained the following results.
\begin{lemma}
The auxiliary quantities $R_n(t),\; r_n(t)$ and the recurrence coefficients $\al_n,\; \bt_n$ satisfy the relations:
\be\label{re1}
R_n(t)=2\al_n-t,
\ee
\be\label{re2}
r_n(t)+r_{n+1}(t)=\lambda-\al_n R_n(t),
\ee
\be\label{re3}
r_n(t)=2\bt_n-n,
\ee
\be\label{re5}
r_n^2(t)-\lambda r_n(t)=\bt_n R_n(t) R_{n-1}(t),
\ee
\be\label{re4}
\sum_{j=0}^{n-1}R_j(t)-t r_n(t)=2\bt_n(R_{n}(t)+R_{n-1}(t)).
\ee
\end{lemma}


We would like to point out that one will obtain the results in Lemma \ref{lemma} by substituting (\ref{re1}) and (\ref{re3}) into (\ref{re2}) and (\ref{re5}) respectively. We now prove Theorem \ref{thm1}.

\begin{proof}[$\mathbf{Proof\; of\; Theorem\; \ref{thm1}}$]
It was shown in \cite{ChenIts} that the orthogonal polynomials $P_n(x)$ satisfy the second-order differential equation
\be\label{ode}
P_n''(x)-\left(\mathrm{v}'(x)+\frac{A_{n}'(x)}{A_{n}(x)}\right)P_n'(x)+\left(B_{n}'(x)-B_{n}(x)\frac{A_{n}'(x)}{A_{n}(x)}
+\sum_{j=0}^{n-1}A_{j}(x)\right)P_n(x)=0,
\ee
which is obtained by eliminating $P_{n-1}(x)$ from the ladder operators (\ref{ra}) and (\ref{lo}).

Next, we will express the coefficients in the above equation in terms of $\al_n$ and $\bt_n$.
Inserting (\ref{re1}) into (\ref{an}) and (\ref{re3}) into (\ref{bn}) gives
\be\label{abz}
A_n(x)=2+\frac{2\al_n-t}{x},\qquad\qquad B_n(x)=\frac{2\bt_n-n}{x}.
\ee
From (\ref{an}) and with the aid of (\ref{re4}), we have
$$
\sum_{j=0}^{n-1}A_{j}(x)=2n+\frac{\sum_{j=0}^{n-1}R_{j}(t)}{x}=2n+\frac{tr_n(t)+2\bt_nR_n(t)+2\bt_nR_{n-1}(t)}{x}.
$$
It follows from (\ref{re5}) that
$$
\bt_nR_{n-1}(t)=\frac{r_n(t)(r_n(t)-\lambda)}{R_n(t)}.
$$
By making use of (\ref{re1}) and (\ref{re3}), we then obtain
\be\label{suma}
\sum_{j=0}^{n-1}A_{j}(x)=2n-\frac{nt-4\al_n\bt_n}{x}-\frac{2(n-2\bt_n)(n+\lambda-2\bt_n)}{x(t-2\al_n)}.
\ee
Substituting (\ref{abz}) and (\ref{suma}) into (\ref{ode}), we establish the theorem.
\end{proof}

\section{Chazy II and $\s$-form of Painlev\'{e} IV}
In this section, we will prove that the recurrence coefficient $\bt_n$ is related to a Chazy type equation, and the sub-leading coefficient $\mathrm{p}(n,t)$ satisfies the Jimbo-Miwa-Okamoto $\s$-form of Painlev\'{e} IV. To prove the results, we introduce the following lemma at first.
\begin{lemma}
The auxiliary quantities $R_n(t)$ and $r_n(t)$ admit the coupled Riccati equations:
\be\label{ri1}
r_n'(t)=\frac{n+r_n(t)}{2}R_n(t)-\frac{r_n^2(t)-\lambda r_n(t)}{R_n(t)},
\ee
\be\label{ri2}
R_n'(t)=\lambda-2r_n(t)-\frac{R_n(t)(t+R_n(t))}{2}.
\ee
\end{lemma}
\begin{proof}
From (\ref{or}) we have
$$
\int_{0}^{\infty}P_{n}^2(x;t)x^\lambda\mathrm{e}^{-x^2+tx}dx=h_{n}(t)
$$
and
$$
\int_{0}^{\infty}P_{n}(x;t)P_{n-1}(x;t)x^\lambda\mathrm{e}^{-x^2+tx}dx=0.
$$
By taking derivatives with respect to $t$, we obtain
\be\label{or1}
\frac{d}{dt}\ln h_{n}(t)=\al_n=\frac{t+R_n(t)}{2}
\ee
and
\be\label{or2}
\frac{d}{dt}\mathrm{p}(n,t)=-\bt_n=-\frac{n+r_n(t)}{2},
\ee
respectively.
Taking account of (\ref{be}), it follows from (\ref{or1}) that
\be\label{equa1}
2\bt_n'=\bt_n R_n(t)-\bt_n R_{n-1}(t).
\ee
Using (\ref{re3}) and (\ref{re5}), we arrive at (\ref{ri1}).

On the other hand, taking a derivative in (\ref{al}) and in view of (\ref{or2}), we find
$$
2\al_n'=1+r_{n+1}(t)-r_n(t)=1+\lambda-\al_nR_n(t)-2r_n(t),
$$
where use has been made of (\ref{re2}) in the second equality. Finally we obtain (\ref{ri2}) with the aid of (\ref{re1}).
\end{proof}
\begin{proof}[$\mathbf{Proof\; of\; Theorem\; \ref{thm2}}$]
Solving $R_n(t)$ from (\ref{ri1}), we have two solutions:
$$
R_n(t)=\frac{r_n'(t)\pm\sqrt{(r_n'(t))^2-2r_n(t)(n+r_n(t))(\lambda-r_n(t))}}{n+r_n(t)}.
$$
Substituting either solution into (\ref{ri2}), we obtain the second-order differential equation satisfied by $r_n(t)$ after removing the square roots:
$$
4\left[r_n''(t)+3r_n^2(t)+2(n-\lambda)r_n(t)-n\lambda\right]^2=t^2\left[(r_n'(t))^2+2r_n^3(t)+2(n-\lambda)r_n^2(t)-2n\lambda r_n(t)\right].
$$
Then equation (\ref{bde}) follows from the relation (\ref{re3}). Under the given transformation, equation (\ref{bde}) turns into the Chazy equation (\ref{ce}).
\end{proof}
\begin{remark}
In the Appendix of the paper \cite{Cosgrove}, Cosgrove gave the relationship between Chazy II system and Painlev\'{e} equations. It is easy to check that our results in Theorem \ref{thm2} are coincident with (A.1) and (A.2) in \cite{Cosgrove} by using Lemma \ref{le}. To be specific, we have $z=\sqrt{2}s$ and $v(z)=\frac{1}{4}q_n'(s)+\frac{1}{4}q_n^2(s)+\frac{1}{2}sq_n(s)-\frac{n}{3}-\frac{\lambda}{6}$, where $q_n(s)$ satisfies the Painlev\'{e} IV equation (\ref{p4}). This corresponds to $A=\frac{1}{2}$, $\epsilon_{1}=1$ and $q=\frac{1}{3}(2n+\lambda)$ (one solution of the cubic equation $4q^3+2\tilde{\al}_1 q+\tilde{\bt}_1=0$ with the values of $\tilde{\al}_1$ and $\tilde{\bt}_1$ given by (\ref{val})) in \cite[(A.2)]{Cosgrove}.
\end{remark}
\begin{remark}
From (\ref{ri2}) we have
$$
r_n(t)=\frac{1}{4}\left(2\lambda-tR_n(t)-R_n^2(t)-2R_n'(t)\right).
$$
Substituting it into (\ref{ri1}), we obtain the second-order differential equation for $R_n(t)$:
$$
8R_n(t)R_n''(t)-4(R_n'(t))^2-3R_n^4(t)-4tR_n^3(t)+(8n+4\lambda+4-t^2)R_n^2(t)+4\lambda^2=0.
$$
Letting $t=2s$ and $R_n(t)=q_n(s)$, we find that $q_n(s)$ satisfies the Painlev\'{e} IV equation
$$
q_n''(s)=\frac{(q_n'(s))^2}{2q_n(s)}+\frac{3}{2}q_n^3(s)+4s\:q_n^2(s)+2(s^2-2n-\lambda-1)q_n(s)-\frac{2\lambda^2}{q_n(s)}.
$$
These results are equivalent to those in Lemma \ref{le}.
\end{remark}
\begin{proof}[$\mathbf{Proof\; of\; Theorem\; \ref{thm3}}$]
Recall that from (\ref{or2}) we have
\be\label{bh}
\bt_n=-\frac{d}{dt}\mathrm{p}(n,t).
\ee
Then equation (\ref{equa1}) becomes
\be\label{equ1}
\bt_n R_n(t)-\bt_n R_{n-1}(t)=-2\frac{d^2}{dt^2}\mathrm{p}(n,t).
\ee
On the other hand, from (\ref{sum}) and (\ref{re1}) we find
\be\label{sum1}
\sum_{j=0}^{n-1}R_j(t)=-2\mathrm{p}(n,t)-nt.
\ee
The combination of (\ref{re3}) and (\ref{bh}) gives
\be\label{rp}
r_n(t)=-2\frac{d}{dt}\mathrm{p}(n,t)-n.
\ee
In view of (\ref{sum1}) and (\ref{rp}), equation (\ref{re4}) turns into
\be\label{equ2}
\bt_n R_n(t)+\bt_n R_{n-1}(t)=t\frac{d}{dt}\mathrm{p}(n,t)-\mathrm{p}(n,t).
\ee

The sum and difference of (\ref{equ1}) and (\ref{equ2}) produce
\be\label{equ3}
2\bt_n R_n(t)=t\frac{d}{dt}\mathrm{p}(n,t)-\mathrm{p}(n,t)-2\frac{d^2}{dt^2}\mathrm{p}(n,t)
\ee
and
\be\label{equ4}
2\bt_n R_{n-1}(t)=t\frac{d}{dt}\mathrm{p}(n,t)-\mathrm{p}(n,t)+2\frac{d^2}{dt^2}\mathrm{p}(n,t),
\ee
respectively. The product of (\ref{equ3}) and (\ref{equ4}) gives
\be\label{pro}
4\bt_n\cdot\bt_nR_n(t) R_{n-1}(t)=\left(t\frac{d}{dt}\mathrm{p}(n,t)-\mathrm{p}(n,t)\right)^2-4\left(\frac{d^2}{dt^2}\mathrm{p}(n,t)\right)^2.
\ee
From (\ref{re5}) and (\ref{rp}) we have
\be\label{brr}
\bt_nR_n(t) R_{n-1}(t)=\left(n+2\frac{d}{dt}\mathrm{p}(n,t)\right)\left(n+\lambda+2\frac{d}{dt}\mathrm{p}(n,t)\right).
\ee
Substituting (\ref{bh}) and (\ref{brr}) into (\ref{pro}), we obtain
\be\label{hnd}
4\left(\frac{d^2}{dt^2}\mathrm{p}(n,t)\right)^2=\left(t\frac{d}{dt}\mathrm{p}(n,t)-\mathrm{p}(n,t)\right)^2+4\frac{d}{dt}\mathrm{p}(n,t)\left(n+2\frac{d}{dt}\mathrm{p}(n,t)\right)
\left(n+\lambda+2\frac{d}{dt}\mathrm{p}(n,t)\right).
\ee
This equation is converted into (\ref{jmo}) under the transformation (\ref{sn}).

Next, we derive the second-order difference equation satisfied by $\mathrm{p}(n,t)$. Substituting (\ref{re1}) and (\ref{re3}) into (\ref{re4}) and using (\ref{sum1}), we have
\be\label{pnt}
\mathrm{p}(n,t)=\bt_n(t-2\al_n-2\al_{n-1}).
\ee
Taking account of (\ref{al}), we can express $\bt_n$ in terms of $\mathrm{p}(n,t)$ and $\mathrm{p}(n\pm 1,t)$:
\be\label{be4}
\bt_n=\frac{\mathrm{p}(n,t)}{t+2\mathrm{p}(n+1,t)-2\mathrm{p}(n-1,t)}.
\ee
Substituting (\ref{re1}) and (\ref{re3}) into (\ref{re5}) gives
\bea\label{tr}
(2\bt_n-n)(2\bt_n-n-\lambda)&=&\bt_n(2\al_n-t)(2\al_{n-1}-t)\nonumber\\
&=&\bt_n\big(2\mathrm{p}(n,t)-2\mathrm{p}(n+1,t)-t\big)\big(2\mathrm{p}(n-1,t)-2\mathrm{p}(n,t)-t\big).
\eea
Inserting (\ref{be4}) into (\ref{tr}), we obtain the second-order difference equation satisfied by $\mathrm{p}(n):=\mathrm{p}(n,t)$:
\bea
&&\left[2\mathrm{p}(n)-n\big(t+2\mathrm{p}(n+1)-2\mathrm{p}(n-1)\big)\right]\left[2\mathrm{p}(n)-(n+\lambda)\big(t+2\mathrm{p}(n+1)-2\mathrm{p}(n-1)\big)\right]\nonumber\\
&=&\mathrm{p}(n)\big(t+2\mathrm{p}(n+1)-2\mathrm{p}(n-1)\big)\big(t+2\mathrm{p}(n+1)-2\mathrm{p}(n)\big)\big(t+2\mathrm{p}(n)-2\mathrm{p}(n-1)\big).\nonumber
\eea
Equation (\ref{dis}) follows from the transformation (\ref{sn}). The proof is complete.
\end{proof}
\begin{remark}
Let $H_n(t)$ be the logarithmic derivative of the Hankel determinant, i.e.,
\be\label{hnt}
H_n(t):=\frac{d}{dt}\ln D_n(t).
\ee
From (\ref{hankel}) we have
$$
H_n(t)=\sum_{j=0}^{n-1}\frac{d}{dt}\ln h_j(t).
$$
Using the first equality in (\ref{or1}) and (\ref{sum}), we obtain
\be\label{hn1}
H_n(t)=-\mathrm{p}(n,t).
\ee
In this case, equation (\ref{jmo}) is equivalent to the result obtained in \cite[Theorem 4.11]{Clarkson}.
\end{remark}

\section{Large $n$ Asymptotics of the Recurrence Coefficients and the Hankel Determinant}
In random matrix theory (RMT), it is known that our Hankel determinant $D_n(t)$ can be viewed as the partition function for the semi-classical Laguerre unitary ensemble \cite[Corollary 2.1.3]{Ismail}. That is,
$$
D_n(t)=\frac{1}{n!}\int_{(0,\infty)^n}\prod_{1\leq i<j\leq n}(x_i-x_j)^2\prod_{k=1}^n x_k^\lambda\mathrm{e}^{-x_k^2+tx_k}dx_k,
$$
where $x_1, x_2, \ldots, x_n$, are the eigenvalues of $n\times n$ Hermitian matrices from the ensemble with the joint probability density function
$$
p(x_1, x_2, \ldots, x_n)=\frac{1}{n!\:D_n(t)}\prod_{1\leq i<j\leq n}(x_i-x_j)^2\prod_{k=1}^n x_k^\lambda\mathrm{e}^{-x_k^2+tx_k}.
$$
See \cite{Mehta,Deift,Forrester2010} for more discussions of this topic.

Dyson's Coulomb fluid approach \cite{Dyson} shows that the collection of eigenvalues (particles) can be approximated as a continuous fluid with a density $\sigma(x)$ supported in $J$, a subset of $\mathbb{R}$, when $n$ is sufficiently large.
It is easy to see that the potential $\mathrm{v}(x)$ in (\ref{po}) is convex for $x\in \mathbb{R}^+$ when $\lambda>0$. In this case, $J$ is a single interval denoted by $(a,b)$; see Chen and Ismail \cite{ChenIsmail} and also \cite[p. 198]{Saff}.

Following \cite{ChenIsmail}, the equilibrium density $\sigma(x)$ is determined by the constrained minimization problem:
$$
\min_{\s}F[\s]\qquad \mathrm{subject}\:\: \mathrm{to}\qquad \int_{a}^{b}\sigma(x)dx=n,
$$
where
$$
F[\s]:=\int_{a}^{b}\s(x)\mathrm{v}(x)dx-\int_{a}^{b}\int_{a}^{b}\s(x)\ln|x-y|\s(y)dxdy
$$
and $\mathrm{v}(x)$ is the potential given by (\ref{po}).

It follows that the density $\s(x)$ satisfies the integral equation
\be\label{ie}
\mathrm{v}(x)-2\int_{a}^{b}\ln|x-y|\s(y)dy=A,\qquad x\in (a,b),
\ee
where $A$ is the Lagrange multiplier that fixes the constraint.
Differentiating the above equation with respect to $x$ gives the singular integral equation
\be\label{sie}
\mathrm{v}'(x)-2P\int_{a}^{b}\frac{\sigma(y)}{x-y}dy=0,\qquad x\in (a,b),
\ee
where $P$ represents the Cauchy principal value. The solution of (\ref{sie}) subject to the boundary condition $\sigma(a)=\sigma(b)=0$ is
\be\label{sigma}
\sigma(x)=\frac{\sqrt{(b-x)(x-a)}}{2\pi^2}P\int_{a}^{b}\frac{\mathrm{v}'(y)}{(y-x)\sqrt{(b-y)(y-a)}}dy
\ee
with a supplementary condition
\be\label{sup1}
\int_{a}^{b}\frac{\mathrm{v}'(x)}{\sqrt{(b-x)(x-a)}}dx=0.
\ee
In addition, using (\ref{sigma}) the normalization condition $\int_{a}^{b}\sigma(x)dx=n$ becomes
\be\label{sup2}
\frac{1}{2\pi}\int_{a}^{b}\frac{x\:\mathrm{v}'(x)}{\sqrt{(b-x)(x-a)}}dx=n.
\ee
The endpoints $a$ and $b$ are determined by (\ref{sup1}) and (\ref{sup2}).
Furthermore, it is shown in \cite{ChenIsmail} that
\begin{subequations}\label{ab}
\be\label{al1}
\al_n=\frac{a+b}{2}+O\left(\frac{\partial^2 A}{\partial t\partial n}\right),
\ee
\be\label{be1}
\bt_n=\left(\frac{b-a}{4}\right)^2\left(1+O\left(\frac{\partial^3 A}{\partial n^3}\right)\right).
\ee
\end{subequations}

Substituting (\ref{po}) for $\mathrm{v}(x)$ into (\ref{sup1}) and (\ref{sup2}) respectively, we get two equations for $a$ and $b$:
\be\label{eq3}
(X-t)^2 Y=\lambda^2,
\ee
\be\label{eq4}
3X^2-2tX-4Y=8n+4\lambda,
\ee
where
$$
X=a+b,\qquad Y=ab.
$$
Eliminating $Y$ from (\ref{eq3}) and (\ref{eq4}), we obtain a quartic equation satisfied by $X$:
$$
(X-t)^2(3X^2-2tX-8n-4\lambda)=4\lambda^2.
$$
This equation has only one positive solution when $n\rightarrow\infty$ and the series expansion reads
\bea
X&=&2 \sqrt{\frac{2n}{3}}+\frac{t}{3}+\frac{t^2+12 \lambda}{12 \sqrt{6n}}-\frac{ t^4+24 \lambda  t^2-288 \lambda ^2}{1152 \sqrt{6}\:n^{3/2}}+\frac{\lambda ^2 t}{8 n^2}\nonumber\\[10pt]
&+&\frac{ t^6+36 \lambda  t^4+9504 \lambda ^2 t^2-13824 \lambda ^3}{55296 \sqrt{6}\:n^{5/2}}+\frac{\lambda ^2 t (t^2-4 \lambda )}{32 n^3}+O(n^{-7/2}).\nonumber
\eea
It follows that
\begin{subequations}\label{ab1}
\bea\label{al2}
\frac{a+b}{2}&=& \sqrt{\frac{2n}{3}}+\frac{t}{6}+\frac{t^2+12 \lambda}{24 \sqrt{6n}}-\frac{ t^4+24 \lambda  t^2-288 \lambda ^2}{2304 \sqrt{6}\:n^{3/2}}+\frac{\lambda ^2 t}{16 n^2}\nonumber\\[10pt]
&+&\frac{ t^6+36 \lambda  t^4+9504 \lambda ^2 t^2-13824 \lambda ^3}{110592 \sqrt{6}\:n^{5/2}}+\frac{\lambda ^2 t (t^2-4 \lambda )}{64 n^3}+O(n^{-7/2}),
\eea
and
\bea\label{be2}
\left(\frac{b-a}{4}\right)^2&=&\frac{X^2-4Y}{16}=\frac{4n+2\lambda+tX-X^2}{8} \nonumber\\[10pt]
&=&\frac{n}{6}+\frac{t\sqrt{n} }{6 \sqrt{6}}+\frac{t^2+6 \lambda}{72}+\frac{ t (t^2+12 \lambda)}{288 \sqrt{6n}}-\frac{\lambda ^2}{16 n}-\frac{ t (t^4+24 \lambda  t^2+3168 \lambda ^2)}{27648 \sqrt{6}\:n^{3/2}}\nonumber\\[10pt]
&-&\frac{\lambda ^2 (3 t^2-4 \lambda)}{128 n^2}+\frac{ t (t^6+36 \lambda  t^4-35424 \lambda ^2 t^2+110592 \lambda ^3)}{1327104 \sqrt{6}\:n^{5/2}}+O(n^{-3}),
\eea
\end{subequations}
where use has been made of (\ref{eq4}) in the second equality.

By using the similar method in \cite{Min2021}, we evaluate the Lagrange multiplier $A$ in the following lemma. The proof will be omitted. The key is that we multiply by $1/\sqrt{(b-x)(x-a)}$ on both sides of equation (\ref{ie}) and then integrate with respect to $x$ from $a$ to $b$.
\begin{lemma}
We have
\bea
A&=&\frac{3a^2+2ab+3b^2}{8}-\frac{(a+b)t}{2}-\lambda\ln\frac{a+b+2\sqrt{ab}}{4}-2n\ln\frac{b-a}{4}\nonumber\\[10pt]
&=&\frac{4n+2\lambda-tX}{4}-\lambda\ln\frac{X^2-tX+2\lambda}{4(X-t)}-n\ln\frac{4n+2\lambda+tX-X^2}{8}.\nonumber
\eea
Then as $n\rightarrow\infty$,
\bea\label{a3}
A&=&-n\ln n+n(1+\ln 6)-t\sqrt{\frac{2n}{3}}-\frac{\lambda}{2}\ln n+\frac{6\lambda\ln 6-t^2}{12}-\frac{ t (t^2+36 \lambda)}{72 \sqrt{6n}}-\frac{\lambda ^2}{2 n}\nonumber\\[10pt]
&+&\frac{ t (t^4+40 \lambda  t^2-1440 \lambda ^2)}{11520 \sqrt{6}\:n^{3/2}}-\frac{\lambda ^2 (3 t^2-20 \lambda)}{96 n^2}+O(n^{-5/2}).
\eea
\end{lemma}
\begin{proof}[$\mathbf{Proof\; of\; Theorem\; \ref{abt}}$]
From (\ref{ab}), (\ref{ab1}) and (\ref{a3}) we see that $\al_n$ and $\bt_n$ have the large $n$ expansion form
\begin{subequations}\label{ab2}
\be\label{al3}
\al_n=\sqrt{\frac{2n}{3}}+\sum_{j=0}^{\infty}\frac{a_j}{n^{\frac{j}{2}}}
\ee
and
\be\label{be3}
\bt_n=\frac{n}{6}+\sum_{j=-1}^{\infty}\frac{b_j}{n^{\frac{j}{2}}},
\ee
\end{subequations}
respectively.
Substituting (\ref{ab2}) into the discrete system (\ref{ab3}) and taking a large $n$ limit, we obtain the expansion coefficients $a_j$ and $b_j$ recursively by equating the powers of $n$:
$$
a_0=\frac{t}{6},\qquad\qquad b_{-1}=\frac{t}{6 \sqrt{6}},\qquad\qquad a_1=\frac{t^2+12 (1+\lambda)}{24 \sqrt{6} },\qquad\qquad b_0=\frac{t^2+6 \lambda}{72},\\[8pt]
$$
$$
a_2=0,\quad b_{1}=\frac{t (t^2+12 \lambda)}{288 \sqrt{6} },\quad a_3=-\frac{t^4+24 t^2(1+\lambda) -48 (6 \lambda ^2-6 \lambda -5)}{2304 \sqrt{6}},\quad b_2=\frac{2-9 \lambda ^2}{144 },
$$
and so on. The theorem is then established.
\end{proof}
\begin{remark}
Very recently, Clarkson and Jordaan studied the generalised Airy polynomials and derived the large $n$ formal asymptotic expansions for the recurrence coefficients; see \cite[Lemma 3.15]{CJ}.
Our approach can be applied to justify the assumption in the proof; see formula (46) in \cite{CJ}.
\end{remark}
\begin{proof}[$\mathbf{Proof\; of\; Theorem\; \ref{hd}}$]
Recall that $\mathrm{p}(n,t)$ can be expressed in terms of $\al_n$ and $\bt_n$ (see (\ref{pnt})):
$$
\mathrm{p}(n,t)=\bt_n(t-2\al_n-2\al_{n-1}).
$$
Substituting (\ref{aln}) and (\ref{ben}) into the above, we obtain (\ref{hne}) after taking a large $n$ limit.
\end{proof}

\begin{proof}[$\mathbf{Proof\; of\; Theorem\; \ref{hkd}}$]
Following the similar development in \cite{Min2021,Min2022} and using the fact
$$
\bt_n=\frac{D_{n+1}(t)D_{n-1}(t)}{D_{n}^2(t)},
$$
we obtain the large $n$ asymptotic expansion of $D_n(t)$:
\bea
\ln D_n(t)&=&\frac{1}{2} n^2 \ln n-\frac{3+2\ln 6}{4}  n^2+\frac{2}{3} \sqrt{\frac{2}{3}}\: n^{3/2}t+\frac{\lambda}{2}     n \ln n+C_1 n\nonumber\\[10pt]
&+&\frac{t (t^2+36 \lambda)\sqrt{n}}{36 \sqrt{6}}+\frac{3 \lambda ^2-1}{6} \ln n +C_2+\frac{t \left[t^4+40 \lambda  t^2+240 (1-6 \lambda ^2)\right]}{5760 \sqrt{6n}}\nonumber\\[10pt]
&-&\frac{(9 \lambda ^2-2) t^2-12 \lambda(5\lambda^2-2)}{288 n}+O(n^{-3/2}),\nonumber
\eea
where $C_1$ and $C_2$ are two undetermined constants independent of $n$.

On the other hand, it is easy to see from (\ref{hnt}) and (\ref{hn1}) that
$$
\ln\frac{D_n(t)}{D_n(0)}=\int_{0}^{t}H_n(u)du=-\int_{0}^{t}\mathrm{p}(n,u)du.
$$
Taking account of (\ref{hne}), we find
\bea\label{dnt}
\ln\frac{D_n(t)}{D_n(0)}&=&\frac{2}{3} \sqrt{\frac{2}{3}}\: n^{3/2} t+\frac{n t^2}{12}+\frac{t (t^2+36 \lambda)\sqrt{n}}{36 \sqrt{6}}+\frac{t^2(t^2+36 \lambda)  }{864}+\frac{t \left[t^4+40 \lambda  t^2+240 (1-6 \lambda ^2)\right]}{5760 \sqrt{6n} } \nonumber\\[10pt]
&+&\frac{\left(2-9 \lambda ^2\right) t^2}{288 n}-\frac{t \left[5 t^6+252 \lambda  t^4+1680 (66 \lambda ^2-17) t^2-60480 \lambda  (8 \lambda ^2-1)\right]}{5806080 \sqrt{6}\: n^{3/2}}\nonumber\\[10pt]
&+&\frac{t^2 \left[(7-27 \lambda ^2) t^2+24 \lambda  (9 \lambda ^2-2)\right]}{13824 n^2}+O(n^{-5/2}).
\eea
Next, we will use the results of Dea$\mathrm{\tilde{n}}$o and Simm \cite{Deano} (see also \cite{Han}) to evaluate $D_n(0)$. By making a simple change of variables, we have
$$
D_n(0)=\frac{n^{n(n+\lambda)/2}}{n!}Z_n(1),
$$
where
$$
Z_n(s):=\int_{(0,\infty)^n}\prod_{1\leq i<j\leq n}(x_i-x_j)^2\prod_{k=1}^n x_k^\lambda\mathrm{e}^{-n[x_k+s(x_k^2-x_k)]}dx_k,\qquad 0\leq s\leq 1.
$$
It follows that
$$
\ln D_n(0)=\frac{1}{2}n^2\ln n+\frac{\lambda}{2}     n \ln n-\ln n-\ln \Gamma(n)+\ln\frac{Z_n(1)}{Z_n(0)}+\ln Z_n(0).
$$
Taking account of (2.31) and (A.3) in \cite{Deano} and with the aid of Stirling's formula (see, e.g., \cite[p. 895]{Gradshteyn}), we obtain
\bea\label{dn0}
\ln D_n(0)&=&\frac{1}{2} n^2 \ln n-\frac{3+2\ln 6}{4}  n^2+\frac{\lambda}{2}     n \ln n+\left[\ln (2 \pi )-\frac{\lambda  (1+\ln 6)}{2} \right]n+\frac{3 \lambda ^2-1}{6} \ln n\nonumber\\[10pt]
&+&\frac{1}{24} \left[48 \zeta'(-1)-24 \ln G(\lambda +1)-12 \lambda ^2 \ln \frac{3}{2}+12 \lambda  \ln (2 \pi )-4 \ln 2+3\ln 3\right]\nonumber\\[10pt]
&+&O(n^{-1}),
\eea
where $\zeta(\cdot)$ is the Riemann zeta function and $G(\cdot)$ is the Barnes $G$-function \cite{Barnes,Voros}.

The combination of (\ref{dnt}) and (\ref{dn0}) shows that
$$
C_1=\frac{t^2}{12}+\ln (2 \pi )-\frac{\lambda  (1+\ln 6)}{2},
$$
$$
C_2=\frac{t^2(t^2+36 \lambda)  }{864}+\frac{1}{24} \left[48 \zeta'(-1)-24 \ln G(\lambda +1)-12 \lambda ^2 \ln \frac{3}{2}+12 \lambda  \ln (2 \pi )-4 \ln 2+3\ln 3\right].
$$
This completes the proof.
\end{proof}

\section*{Acknowledgments}
The work of C. Min was partially supported by the National Natural Science Foundation of China under grant number 12001212, by the Fundamental Research Funds for the Central Universities under grant number ZQN-902 and by the Scientific Research Funds of Huaqiao University under grant number 17BS402. The work of
Y. Chen was partially supported by the Macau Science and Technology Development Fund under grant number FDCT 0079/2020/A2.
\section*{Conflicts of Interest}
The authors have no conflicts of interest to declare that are relevant to the content of this article.

\section*{Data Availability Statements}
Data sharing not applicable to this article as no datasets were generated or analysed during the current study.


\begin{thebibliography}{}
%
%
\bibitem{Alhaidari}
{A. D. Alhaidari}, {Series solutions of Heun-type equation in terms of orthogonal polynomials}, {\em J. Math. Phys.} {\bf 59} ({2018}) {113507}.
\bibitem{Barnes}
{E. W. Barnes}, The theory of the $G$-function, {\em Quart. J. Pure Appl. Math.} {\bf 31} ({1900}) {264--314}.
\bibitem{Boelen}
{L. Boelen} and {W. Van Assche}, {Discrete Painlev\'{e} equations for recurrence coefficients of semiclassical Laguerre polynomials}, {\em Proc. Amer. Math. Soc.} {\bf 138} ({2010}) {1317--1331}.
\bibitem{Chazy1909}
{J. Chazy}, {Sur les \'{e}quations diff\'{e}rentielles du second ordre \`{a} points critiques fixes}, {\em C. R. Acad. Sci. Paris} {\bf 148} ({1909}) {1381--1384}.
\bibitem{Chazy1911}
{J. Chazy}, {Sur les \'{e}quations diff\'{e}rentielles du troisi\`{e}me ordre et d'ordre sup\'{e}rieur dont l'int\'{e}grale g\'{e}n\'{e}rale a ses points critiques fixes}, {\em Acta Math.} {\bf 34} ({1911}) {317--385}.
\bibitem{CFZ}
{Y. Chen}, {G. Filipuk} and {L. Zhan}, {Orthogonal polynomials, asymptotics, and Heun equations}, {\em J. Math. Phys.} {\bf 60} ({2019}) {113501}.
\bibitem{ChenIsmail}
{Y. Chen} and {M. E. H. Ismail}, {Thermodynamic relations of the Hermitian matrix ensembles}, {\em J. Phys. A: Math. Gen.} {\bf 30} ({1997}) {6633--6654}.
\bibitem{ChenIts}
{Y. Chen} and {A. Its}, {Painlev\'{e} III and a singular linear statistics in Hermitian random matrix ensembles, I}, {\em J. Approx. Theory} {\bf 162} ({2010}) {270--297}.
\bibitem{Clarkson}
{P. A. Clarkson} and {K. Jordaan}, {The relationship between semiclassical Laguerre polynomials and the fourth Painlev\'{e} equation}, {\em Constr. Approx.} {\bf 39} ({2014}) {223--254}.
\bibitem{CJ}
{P. A. Clarkson} and {K. Jordaan}, {Generalised Airy polynomials}, {\em J. Phys. A: Math. Theor.} {\bf 54} ({2021}) 185202 (28pp).
\bibitem{Clarkson3}
{P. A. Clarkson}, {K. Jordaan} and {A. Kelil}, {A generalized Freud weight}, {\em Stud. Appl. Math.} {\bf 136} ({2016}) {288--320}.
\bibitem{Cosgrove}
{C. M. Cosgrove}, {Chazy's second-degree Painlev\'{e} equations}, {\em J. Phys. A: Math. Gen.} {\bf 39} ({2006}) {11955--11971}.
\bibitem{Dai}
{D. Dai} and {L. Zhang}, {Painlev\'{e} VI and Hankel determinants for the generalized Jacobi weight}, {\em J. Phys. A: Math. Theor.} {\bf 43} ({2010}) {055207 (14pp)}.
\bibitem{Deano}
{A. Dea$\mathrm{\tilde{n}}$o} and {N. J. Simm}, {On the probability of positive-definiteness in the gGUE via semi-classical Laguerre polynomials}, {\em J. Approx. Theory} {\bf 220} ({2017}) {44--59}.
\bibitem{Deift}
{P. Deift}, {\em Orthogonal Polynomials and Random Matrices: A Riemann-Hilbert Approach}, {Courant Lecture Notes 3}, {New York University}, {New York}, {1999}.
\bibitem{Dyson}
{F. J. Dyson}, {Statistical theory of the energy levels of complex systems, I, II, III}, {\em J. Math. Phys.} {\bf 3} ({1962}) {140--156, 157--165, 166--175}.
\bibitem{Filipuk}
{G. Filipuk}, {W. Van Assche} and {L. Zhang}, {The recurrence coefficients of semi-classical Laguerre polynomials and the fourth Painlev\'{e} equation}, {\em J. Phys. A: Math. Theor.} {\bf 45} ({2012}) {205201 (13 pp)}.
\bibitem{Forrester2010}
{P. J. Forrester}, {\em Log-Gases and Random Matrices}, {Princeton University Press}, {Princeton}, {2010}.
\bibitem{Gradshteyn}
{I. S. Gradshteyn} and {I. M. Ryzhik}, {\em Table of Integrals, Series, and Products: Seventh Edition}, {Academic Press}, {New York}, {2007}.
\bibitem{Gromak}
{V. I. Gromak}, {I. Laine} and {S. Shimomura}, {\em Painlev\'{e} Differential Equations in the Complex Plane}, {Walter de Gruyter}, {Berlin}, {2002}.
\bibitem{Han}
{P. Han} and {Y. Chen}, {The recurrence coefficients of a semi-classical Laguerre polynomials and the large $n$ asymptotics of the associated Hankel determinant}, {\em Random Matrices: Theor. Appl.} {\bf 6} ({2017}) {1740002} (20 pages).
\bibitem{Ismail}
{M. E. H. Ismail}, {\em Classical and Quantum Orthogonal Polynomials in One Variable}, {Encyclopedia of Mathematics and its Applications 98}, {Cambridge University Press}, {Cambridge}, {2005}.
\bibitem{Jimbo1981}
{M. Jimbo} and {T. Miwa}, {Monodromy preserving deformation of linear ordinary differential equations with rational coefficients. II}, {\em Physica D} {\bf 2} ({1981}) {407--448}.
\bibitem{Lyu2017}
{S. Lyu} and {Y. Chen}, {The largest eigenvalue distribution of the Laguerre unitary ensemble}, {\em Acta Math. Sci.} {\bf 37} ({2017}) {439--462}.
\bibitem{MNR}
{A. P. Magnus}, {F. Ndayiragije} and {A. Ronveaux}, {About families of orthogonal polynomials satisfying Heun's differential equation}, {\em J. Approx. Theory} {\bf 263} ({2021}) {105522}.
\bibitem{Mehta}
{M. L. Mehta}, {\em Random Matrices}, {3rd edn.}, {Elsevier}, {New York}, {2004}.
\bibitem{Min2019}
{C. Min} and {Y. Chen}, {Painlev\'{e} transcendents and the Hankel determinants generated by a discontinuous Gaussian weight}, {\em Math. Meth. Appl. Sci.} {\bf 42} ({2019}) {301--321}.
\bibitem{Min2021}
{C. Min} and {Y. Chen}, {Differential, difference, and asymptotic relations for Pollaczek-Jacobi type orthogonal polynomials and their Hankel determinants}, {\em Stud. Appl. Math.} {\bf 147} ({2021}) {390--416}.
\bibitem{Min2022}
{C. Min} and {Y. Chen}, {Hankel determinant and orthogonal polynomials for a perturbed Gaussian weight: from finite $n$ to large $n$ asymptotics}, {arXiv: 2203.10526} {(28 pp)}.
\bibitem{Ronveaux}
{A. Ronveaux}, {\em Heun's Differential Equations}, {Oxford Science Publications}, {Oxford}, {1995}.
\bibitem{Saff}
{E. B. Saff} and {V. Totik}, {\em Logarithmic Potentials with External Fields}, {Springer}, {Berlin}, {1997}.
\bibitem{Sogo}
{K. Sogo}, {Time-dependent orthogonal polynomials and theory of soliton-applications to matrix model, vertex model and level statistics}, {\em J. Phys. Soc. Japan} {\bf 62} ({1993}) {1887--1894}.
\bibitem{Szego}
{G. Szeg\"{o}}, {\em Orthogonal Polynomials}, {4th edn.}, {Amer. Math. Soc.}, {Providence, RI}, {1975}.
\bibitem{VanAssche}
{W. Van Assche}, {\em Orthogonal Polynomials and Painlev\'{e} Equations}, {Australian Mathematical Society Lecture Series 27}, {Cambridge University Press}, {Cambridge}, {2018}.
\bibitem{Voros}
{A. Voros}, Spectral functions, special functions and the Selberg zeta function, {\em Commun. Math. Phys.} {\bf 110} ({1987}) {439--465}.
\bibitem{Witte}
{N. S. Witte}, {P. J. Forrester} and {C. M. Cosgrove}, {Gap probabilities for edge intervals in finite Gaussian and Jacobi unitary matrix ensembles}, {\em Nonlinearity} {\bf 13} ({2000}) {1439--1464}.
\end{thebibliography}
\end{document}